\documentclass[12pt,a4paper]{article}
\usepackage{amssymb,amsmath,amsthm}
\usepackage{graphicx,color}
\usepackage{hyperref}
\usepackage[dvipsnames]{xcolor}

\newcommand\cC{{\mathcal C}}
\newcommand\cD{{\mathcal D}}

\newcommand\cF{{\mathcal F}}

\newcommand{\floor}[1]{\left\lfloor{#1}\right\rfloor}

\newcommand\cH{{\mathcal H}}
\newcommand\cI{{\mathcal I}}

\newcommand\cQ{{\mathcal Q}}

\newcommand\cU{{\mathcal U}}

\theoremstyle{plain}
\newtheorem{theorem}{Theorem}[section]
\newtheorem{lemma}[theorem]{Lemma}
\newtheorem{corollary}[theorem]{Corollary}

\newtheorem{proposition}[theorem]{Proposition}

\theoremstyle{definition}

\newtheorem{defn}[theorem]{Definition}

\newtheorem{claim}[theorem]{Claim}

\newtheorem{question}[theorem]{Question}
\newtheorem*{rem}{Remark}

\newcommand\cref[1]{Corollary~\ref{cor:#1}}

\title{Rainbow Ramsey problems for the Boolean lattice}
\begin{document}

\author{Fei-Huang Chang$^a$, D\'{a}niel Gerbner$^b$, Wei-Tian Li$^c$, \\
Abhishek Methuku$^d$, 
D\'{a}niel T. Nagy$^b$, Bal\'{a}zs Patk\'{o}s$^b$, M\'{a}t\'{e} Vizer$^b$ \\
\small $^a$ Division of Preparatory Programs for Overseas Chinese Students, \\
\small National Taiwan Normal University New Taipei City, Taiwan
 \\
\small $^b$Alfr\'ed R\'enyi Institute of Mathematics, Hungarian Academy of Sciences\\
\small $^c$ Department of Applied Mathematics, National Chung Hsing University, Taichung 40227, Taiwan
\\
\medskip
\small $^d$ Central European University, Budapest and \'Ecole Polytechnique F\'ed\'erale de Lausanne  \\
\small {\tt cfh@ntnu.edu.tw}, \texttt{\{gerbner,nagydani,patkos\}@renyi.hu}, \\
\small {\tt weitianli@nchu.edu.tw}, \texttt{\{abhishekmethuku,vizermate\}@gmail.com}
}
\maketitle

\begin{abstract}
We address the following rainbow Ramsey problem: For posets $P,Q$ what is the smallest number $n$ such that any coloring of the elements of the Boolean lattice $B_n$ either admits a monochromatic copy of $P$ or a rainbow copy of $Q$. We consider both weak and strong (non-induced and induced) versions of this problem.
We also investigate related problems on (partial) $k$-colorings of $B_n$ that do not admit rainbow antichains of size $k$.
\end{abstract}
\section{Introduction}


In this paper we consider rainbow Ramsey-type problems for posets. Given posets $P$ and $Q$, we say that $X \subseteq Q$ is a \textit{weak copy} of $P$ if there is a bijection $\alpha:P\rightarrow X$ such that $p\le_P p'$ implies $\alpha(p)\le_Q \alpha(p')$. If $\alpha$ has the stronger property that $p\le_P p'$ holds if and only if $\alpha(p)\le_Q \alpha(p')$, then $X$ is a \textit{strong} or \textit{induced copy of} $P$. A copy $X$ of $P$ is \textit{monochromatic} with respect to a coloring $\phi:Q\rightarrow \mathbb{N}$ if $\phi(q)=\phi(q')$ for all $q,q'\in X$ and \textit{rainbow} if $\phi(q)\neq \phi(q')$ for all $q,q'\in X$. We will be looking for monochromatic and/or rainbow copies of some posets in the Boolean lattice $B_n$, the subsets of an $n$-element set ordered by inclusion. The set of elements of $B_n$ corresponding to sets of the same size is called a \textit{level} of $B_n$.

\begin{defn}

The \textit{weak Ramsey number} $R(P_1,P_2,\dots,P_k)$ is the smallest number $n$ such that for any coloring of the elements of $B_n$ with $k$ colors, say $1,2, \ldots,k$ there is a monochromatic copy of the poset $P_i$ in color $i$ for some $1 \le i \le k$. 
We simply write $R_{k}(P)$ for $R(P_1,P_2,\dots,P_k),$ if $P_1= \ldots = P_k = P$.
We define the \textit{strong Ramsey number} $R^{*}(P_1,P_2,\dots,P_k)$ and $R^{*}_{k}(P)$ for strong copies of posets analogously.\end{defn}

Ramsey theory of posets is an old and well investigated topic, see e.g., \cite{KT1987,T1999}. However the study of Ramsey problems in the Boolean lattice was initiated only recently: Weak Ramsey numbers were studied by Cox and Stolee \cite{CS} and strong Ramsey numbers were investigated by Axenovich and Walzer \cite{AW}. In addition, some results in the latter one were improved by Lu and Thompson \cite{LT}.
\vspace{2mm}

In this article we study rainbow Ramsey numbers for the Boolean lattice. 

\begin{defn}
For two posets $P,Q$ the \textit{weak (or not necessarily induced) rainbow Ramsey number} $RR(P,Q)$ is the minimum number $n$ such that any coloring of $B_n$ admits either a monochromatic weak copy of $P$ or a rainbow weak copy of $Q$.
\textit{Strong (or induced) rainbow Ramsey number} can be defined analogously and is denoted by $RR^{*}(P,Q)$.
\end{defn}


\noindent 
Rainbow Ramsey numbers for graphs have been intensively studied (they are sometimes called constrained Ramsey numbers or Gallai-Ramsey numbers), for a recent survey see \cite{FMO}. 
The results on the rainbow Ramsey number for Boolean posets are sporadic \cite{CCLL,JLM}. Nevertheless, the following easy observation connects the (usual) Ramsey numbers to the rainbow Ramsey numbers.

\begin{proposition}\label{easy}
For any pair $P$ and $Q$ of posets we have

\vspace{3mm}

\textbf{(i)} $RR(P,Q) \ge R_{ |Q|-1}(P)$, and

\vspace{2mm}

\textbf{(ii)} $RR^{*}(P,Q) \ge R^{*}_{ |Q|-1}(P)$.
\end{proposition}

\begin{proof}
To see \textbf{(i)} observe that if a coloring $\phi$ uses at most $|Q|-1$ colors, then clearly it cannot contain a rainbow weak copy of $Q$. Therefore any such coloring showing $R_{ |Q|-1}(P)>n$  also shows $RR(P,Q)>n$. An identical proof with strong copies implies \textbf{(ii)}.
\end{proof}

In this paper, we show many examples of posets $P, Q$ for which the inequality in \textbf{(i)} of Proposition \ref{easy} holds with equality, while in Section 3, we show another example of posets $P, Q$ for which \textbf{(ii)} of Proposition \ref{easy} holds with strict inequality. Unfortunately, we do not know whether there exists posets $P,Q$ for which \textbf{(i)} holds with strict inequality.

\vskip 0.2cm

Many of the tools used in \cite{AW,CS} come from the related Tur\'an-type problem, the so-called forbidden subposet problem. Let us introduce some terminology. For a poset $P$, a family $\cF\subseteq B_n$ of sets is called (induced) \textit{$P$-free} if $\cF$ does not contain a weak (strong) copy of $P$. The size of the largest (induced) $P$-free family in $B_n$ is denoted by $La(n,P)$ ($La^*(n,P)$). For a poset $P$ we denote by $e(P)$ the maximum number $m$ such that for any $n$ the union of any consecutive $m$ levels of $B_n$ is $P$-free. The analogous strong parameter is denoted by $e^*(P)$. The most widely believed conjecture \cite{GL} in the area of forbidden subposet problems states that for any poset $P$ we have $$\lim_{n\rightarrow \infty}\frac{La(n,P)}{\binom{n}{\lfloor n/2\rfloor}}=e(P) \textrm{ and } \lim_{n\rightarrow \infty}\frac{La^*(n,P)}{\binom{n}{\lfloor n/2\rfloor}}=e^*(P).$$ It is worth noting that this conjecture is already wide open for a very simple poset called the diamond poset $D_2$ (defined on four elements $a,b,c,$ and $d$ with relations $ a < b, c$ and $b, c  < d$). See \cite{diamond} for the best known bounds in this direction.

For a family $\cF\subseteq B_n$ of sets, its \textit{Lubell-mass} is $\lambda_n(\cF)=\sum_{F\in \cF}\frac{1}{\binom{n}{|F|}}$. For a poset $P$, we define $\lambda_n(P)$ to be the maximum value of $\lambda_n(\cF)$ over all $P$-free families $\cF\subseteq B_n$ and $\lambda_{max}(P)$ is defined to be $\sup_n\lambda_n(P)$. Its finiteness follows from the fact that every poset $P$ is a weak subposet of $C_{|P|}$ (where $C_l$ denotes the $l$-chain, the totally ordered set of size $l$) and the $k$-LYM-inequality stating that $\lambda_n(\cF)\le k$ for any $C_{k+1}$-free family $\cF\subseteq B_n$. Analogously, $\lambda^*_n(P)$ is the maximum value of $\lambda_n(\cF)$ over all induced $P$-free families $\cF\subseteq B_n$ and $\lambda_{max}^*(P)$ is defined to be $\sup_n\lambda_n^*(P)$. It was proved to be finite by M\'eroueh \cite{M}. 

Observe that, by definition of $e(P)$ and $e^*(P)$, we have $e(P)\le \lambda_n(P)$ and $e^*(P)\le \lambda^*_n(P)$ for every poset $P$ and integer $n\ge e(P)$ or $n\ge e^*(P)$. We say that a poset is \textit{uniformly Lubell-bounded} if $e(P)\ge \lambda_n(P)$ holds for all positive integers $n$. Similarly, a poset is \textit{uniformly induced Lubell-bounded} if $e^*(P)\ge \lambda^*_n(P)$ holds for all positive integers $n$. 
An instance of posets eqipped with this property is the class of chain posets $C_l$. 
For $k \ge 2$ the \textit{generalized diamond} poset $D_k$ consists of $k+2$ elements $a,b_1,b_2,\dots,b_k,c$ with relations $a<b_i<c$ for $1\le i\le k$.
Griggs, Li and Lu \cite{GLL} proved that infinitely many of the $D_k$'s are uniformly Lubell-bounded and Patk\'os \cite{P} proved that an overlapping but distinct and infinite subset of the $D_k$'s is uniformly induced Lubell-bounded.
For more uniformly Lubell-bounded posets, see~\cite{GLi2}.

In \cite{AW} and \cite{CS}, it was observed that if $P$ is uniformly Lubell-bounded or uniformly induced Lubell-bounded, then $R_{k}(P)=k\cdot e(P)$ or $R^{*}_{k}(P)=k\cdot e^*(P)$ holds, respectively.


Our main result concerning weak rainbow Ramsey numbers extends the above observation.

\begin{theorem}\label{main}
Let $P$ be a uniformly Lubell-bounded poset and $\cF\subseteq B_n$ be  a family of sets with $\lambda_n(\cF)>e(P)(k-1)$. Then any coloring of $\phi:\cF\rightarrow \mathbb{N}$ admits either a monochromatic weak copy of $P$ or a rainbow copy of $C_k$.
\end{theorem}

As $\lambda_n(B_n)=n+1$, one direction of the following equality is a direct consequence of Theorem~\ref{main}.

\begin{corollary}\label{cormain}
If $P$ is uniformly Lubell-bounded, then $RR(P,Q)=e(P)(|Q|-1)$ holds for any poset $Q$.
\end{corollary}

Let $n=(|Q|-1)e(P)-1$.  The lower bound $RR(P,Q)>n$ in Corollary \ref{cormain} follows from coloring $B_n$ so that the color classes form a partition of the levels of $B_n$ into $|Q|-1$ intervals, each of size $e(P)$. As we use only $|Q|-1$ colors, we avoid rainbow copies of $Q$ and by definition of $e(P)$ we avoid monochromatic copies of $P$.

For strong copies the same coloring yields the same lower bound  $RR^{*}(P,Q)\ge e^{*}(P)(|Q|-1)$, 
but one can easily observe that in most cases this trivial lower bound can be improved by slightly modifying the above coloring: 
If $Q$ does not have a unique smallest element, then one can color $\emptyset$ with an otherwise unused color $i$. Since no other sets are colored $i$ it does not help creating a strong monochromatic copy of $P$, and since $Q$ does not have a unique smallest element, it does not help creating a strong rainbow copy of $Q$. Therefore one can introduce the following function. For any poset $Q$ let $f(Q)=0$ if $Q$ has both a unique largest and a unique smallest element, let $f(Q)=2$, if $Q$ has neither largest nor  smallest element, and let $f(Q)=1$ otherwise. One obtains $RR^{*}(P,Q)\ge e^*(P)(|Q|-1)+f(Q)$ for all posets $P$ and $Q$. 
For this lower bound,
the strong version of Corollary~\ref{cormain} would be expected for $P$ being uniformly induced Lubell-bounded. 
Nonetheless, we will show the above inequality is strict when $P=C_2$, the chain of two elements, and $Q=A_k$, the antichain of size $k$ in Section 3. So we ask the following question.

\begin{question}\label{firstQ}
For which uniformly induced Lubell-bounded posets $P$, one has 
\begin{equation}\label{RRPQ}
RR^{*}(P,Q)=e^*(P)(|Q|-1)+f(Q)
\end{equation}for every poset $Q$?
\end{question}

Despite the above counterexample to Equality (\ref{RRPQ}), we prove that it holds for most uniformly induced Lubell-bounded posets $P$ and $Q=A_3$. Indeed, we have a general upper bound for $RR^{*}(P,A_k)$ 
for any poset $P$ and $k\ge 2$.

\begin{theorem}\label{genstrong}
Given an integer $k\ge 2$, let $m_k=\min\{m: \binom{m}{\lfloor m/2\rfloor }\ge k\}$. 
For any poset $P$ we have $$RR^{*}(P,A_k)\le \lfloor(k-1)\lambda^*_{max}(P)\rfloor+m_k.$$
Moreover, if $P$ is not $C_1$ or $C_2$, then  
we have \[RR^{*}(P,A_3)\le \lfloor2\lambda^*_{max}(P)\rfloor+2.\] 
\end{theorem}

Since $\lambda^{*}_{max}(P)=e^{*}(P)$ for every uniformly induced Lubell-bounded poset $P$, 
we have the next corollary immediately from the moreover part of Theorem \ref{genstrong}.

\begin{corollary}\label{strong}
For every uniformly induced Lubell-bounded poset $P$ other than $C_1$ or $C_2$ we have 
\[RR^{*}(P,A_3)=2+2e^*(P).\]
\end{corollary}







\vskip0.2cm

\textbf{Structure of the paper.} The remainder of the paper is organized as follows: Theorem \ref{main} and other results on weak copies are proved in Section 2. Section 3 contains the proofs of the counterexample to Equation (\ref{RRPQ}) and Theorem \ref{genstrong}. In Section 4, we introduce four extremal functions $F$, $F'$, $G$ and $G'$ related to colorings of $B_n$ and obtain some bounds on their values.

\vskip0.2cm

\textbf{Notation.} For a set $F$ we write $\cU_F=\cU_{n,F}=\{G\subseteq [n]:F\subseteq G\}$, $\cD_F=\cD_{n,F}=\{G\subseteq [n]:G\subseteq F\}$, $\cI_F=\cI_{n,F}=\cU_{n,F}\cup \cD_{n,F}$. 
For sets $F\subseteq H$ we write $B_{F,H}=\{G: F\subseteq G\subseteq H\}$. 
For integers $0\le a\le b\le n$ we write $\lambda_n(B_{a,b})=\lambda_n(B_{F,H})$ for some $F\subseteq H\subseteq [n]$ with $|F|=a,|H|=b$. 
Let $B^{-}_n$ and $B^{-}_{F,H}$ denote the \emph{truncated} Boolean lattices obtained by removing the smallest and the largest element of the cubes $B_n$ and $B_{F,H}$ respectively. For a coloring $\phi:B_n\rightarrow \mathbb{Z}^+$ let $\|\phi\|$ denote the number of colors used by $\phi$. For a coloring $\phi: B_n\rightarrow \mathbb{Z}^+$ and a positive integer $i$ let $\cH_i=\cH_{\phi,i}=\{F\subseteq [n]: \phi(F)=i\}$.
We use $\binom{n}{\le k}$ to denote $\sum_{j=0}^k\binom{n}{j}$. All logarithms are of base 2 in this paper.

\section{Weak copies}




In this section, we prove Theorem \ref{main} and some other results on weak Ramsey and rainbow Ramsey numbers. We start with a couple of definitions. 

Let us denote by $\mathbf{C}_n$ the set of all maximal chains in $B_n$. For a family $\cF\subseteq B_n$ and set $F\in \cF$ we define $\mathbf{C}_{n,F}=\mathbf{C}_{n,F,\cF}$ to be the set of those maximal chains $\cC\in \mathbf{C}_n$ for which the largest set of $\cF\cap \cC$ is $F$. Then the \textit{max-partition} of $\mathbf{C}_n$ consists of the blocks $\mathbf{C}_{n,F}$ for each $F\in \cF$ and $\mathbf{C}_{n,-}$ which contains all maximal chains $\cC$  with $\cF\cap \cC=\emptyset$.

 The Lubell mass $\lambda_n(\cF)=\sum_{F\in\cF}\frac{1}{\binom{n}{|F|}}$ is the average number of sets of $\cF$ in a maximal chain $\cC$ chosen uniformly at random from $\mathbf{C}_n$. As observed by  Griggs and Li \cite{GLi} if we condition on the largest set $F$ in $\cF\cap \cC$, then we obtain
$$\lambda_n(\cF)=\sum_{F\in\cF}\frac{|\mathbf{C}_{n,F}|}{n!}\lambda_{|F|}(\cD_F\cap \cF).$$

\begin{proof}[Proof of Theorem \ref{main}]
We proceed by induction on $k$. The base case $k=1$ is trivial as any colored set forms a ``rainbow'' copy of $C_1$. Suppose the statement is proven for $k-1$ and let $\cF\subseteq B_n$ be a family of sets with $\lambda_n(\cF)>e(P)(k-1)$. Let us fix a coloring $\phi:\cF\rightarrow \mathbb{N}$ and let us consider the max-partition $\{\mathbf{C}_{n,F}: F\in \cF\}\cup \{\mathbf{C}_{n,-}\}$. Using  $$\lambda_n(\cF)=\sum_{F\in\cF}\frac{|\mathbf{C}_{n,F}|}{n!}\lambda_{|F|}(\cD_F\cap \cF)$$ we obtain a set $F\in \cF$ with $\lambda_{|F|}(\cD_F\cap \cF)>e(P)(k-1)$. Let $\cF_1=\{G\in \cD_F: \phi(G)=\phi(F)\}$. If $\cF_1$ contains a weak copy of $P$, then we are done as, by definition, $\cF_1$ is monochromatic. Otherwise, as $P$ is uniformly Lubell-bounded, we have $\lambda_{|F|}(\cF_1)\le e(P)$ and thus
\[
\lambda_{|F|}((\cD_F\cap \cF)\setminus\cF_1)> e(P)(k-1)-e(P)=e(P)(k-2).
\]
Applying our inductive hypothesis to $(\cD_F\cap \cF)\setminus \cF_1$ we either obtain a monochromatic weak copy of $P$ or a rainbow copy of $C_{k-1}$. As all sets in $(\cD_F\cap \cF)\setminus \cF_1$ are colored differently than $F$, we can extend the rainbow copy of $C_{k-1}$ to a rainbow copy of $C_k$ by adding $F$.
\end{proof}

\begin{rem}
Note that a simple modification of the above proof shows that if $P$ is a uniformly induced  Lubell-bounded poset and $\cF\subseteq B_n$ is a family of sets with $\lambda_n(\cF)>e^*(P)(k-1)$, then any coloring of $\phi:\cF\rightarrow \mathbb{N}$ admits either a monochromatic strong copy of $P$ or a rainbow copy of $C_k$, and therefore $RR^{*}(P,C_k)=e^*(P)(k-1)$ holds.
\end{rem}

The equality in Proposition~\ref{easy} (i) holds for uniformly Lubell-bounded posets $P$ and any posets $Q$. To find posets $P$ and $Q$ with $RR(P,Q)>R_{|Q|-1}(P)$, we have to choose a non-uniformly Lubell-bounded poset as $P$. However, regardless of $P$, Proposition~\ref{easy} (i) still holds with equality
if  $Q$ is one of the following posets:
For $r\ge 2$ the {\em $r$-fork} poset $V_r$ consists of a minimal element and $r$ other elements that form an antichain. Similarly, for $s\ge 2$ {\em the $s$-broom} poset $\Lambda_s$ consists of a maximal element and $s$ other elements that form an antichain.


\begin{proposition}
For any poset $P$ we have 

\textbf{(i)} $RR(P,V_r)=R_{r}(P)$, and

\textbf{(ii)} $RR(P,\Lambda_s)=R_{s}(P)$.
\end{proposition}

\begin{proof}
By Proposition~\ref{easy} $RR(P,V_r)\ge R_{r}(P)$. Let $n=R_{r}(P)$.
Any coloring $\phi:B_n\rightarrow \mathbb{N}$ with $\|\phi\|\ge r+1$ contains a rainbow weak copy of $V_r$: the empty set and one representative from each of any other $r$ color classes.

The proof of (ii) is similar by taking the universal set $[n]$ and one representative from each of any $s$ other color classes if $\|\phi\|\ge s+1$.
\end{proof}

If $P$ and $Q$ are both fork posets, then we have $RR(V_r,V_k)=R_k(V_r)$.
In our next result, we manage to determine this value asymptotically for fixed $r$. 
Let us write $f_k(r)=R_{k}(V_r)$ for simplicity. 
A simple $k$-coloring of $B_n$ avoiding monochromatic weak copies of $V_r$  is to color sets of the same size with the same color, and color classes should consist of consecutive levels. Formally, let $i_1,i_2,\dots, i_k$ be positive integers with $\sum_{j=1}^ki_j=n+1$ and consider the coloring $\phi(F)=h$ if and only if $\sum_{j=1}^{h-1}i_j\le |F|< \sum_{j=1}^hi_j$. (The empty sum equals 0, so $\phi(F)=1$ if and only if $|F|< i_1$ holds.) Let us call such a coloring $\phi$ \textit{consecutive level $k$-coloring} and let us define $g_k(r)$ to be the smallest integer $n$ such that any consecutive level $k$-coloring of $B_n$ admits a monochromatic weak copy of $V_r$. By definition, we have $g_k(r)\le f_k(r)$.

For $c \in (0,1)$ let $h(c)=-c\log c-(1-c)\log(1-c)$, the \textit{binary entropy function}. Note that for $c \in (0,1)$ and $n$ large enough we have $$ \frac{1}{\sqrt{n}} 2^{n h(c)}\le \binom{n}{\lfloor c n \rfloor} \le 2^{n h(c)}.$$
 In the proof we omit floor and ceiling signs for simplicity.

\begin{theorem}\label{weakvee}
For any positive integer $k$ there exists a constant $c_k$ such that 
\[\lim_{r\rightarrow\infty}\frac{g_k(r)}{\log r}=\lim_{r\rightarrow\infty}\frac{f_k(r)}{\log r}=c_k.\]
Moreover, $c_1=1$ and the sequence $\{c_k\}_{k=1}^{\infty}$ satisfies the equality $c_{k+1}h(\frac{c_{k+1}-c_k}{c_{k+1}})=1$ for any $k\ge 1$. 
\end{theorem}

\begin{proof}
The proof is based on the following simple observations.

\begin{claim}\label{veeclaim}
For any $k\ge 1$ and $r\ge 1$ we have

\vspace{2mm}

\textbf{(i)} $f_{k+1}(r)\le f_k(2r-1)+\min\{a:\binom{a+f_k(2r-1)}{\le a}> r\}$,

\vspace{1mm}

\textbf{(ii)} $g_{k+1}(r)\ge g_k(r)+\max\{a:\binom{a+g_k(r)}{\le a}\le r\}+1$.

\end{claim}

\begin{proof}[Proof of the claim]
Let $N=f_k(2r-1)+\min\{a:\binom{a+f_k(2r-1)}{\le a}> r\}$ and let us consider a coloring $\phi:B_N\rightarrow [k+1]$. Without loss of generality we may assume $\phi(\emptyset)=k+1$ for the empty set $\emptyset$. Assume first that there exists a set $F\in B_N$ with $|F|\le \min\{a:\binom{a+f_k(2r-1)}{\le a}> r\}$ and $\phi(F)\neq k+1$. Then consider the $k$-coloring $\phi':B_{F,[N]}\rightarrow [k]$ defined by $\phi'(G)=\phi(G)$, if $\phi(G)\in [k]$ and $\phi'(G)=\phi(F)$ otherwise. As $|F|\le \min\{a:\binom{a+f_k(2r-1)}{\le a}> r\}$, $\phi'$ admits a monochromatic weak copy $C$ of $V_{2r-1}$ in $B_{F,[N]}$. If its color is not $\phi(F)$, then its elements have hte same color in $\phi$, thus $C$ is a monochromatic weak copy of $V_{2r-1}$ with respect to $\phi$. If the color of $C$ is $\phi(F)$ and $C$ contains at least $r$ sets that were colored $k+1$ in the coloring $\phi$, then together with the empty set, they form a monochromatic weak copy of $V_r$ with respect to $\phi$. Otherwise $C$ contains at least $r+1$ sets, including $F$, that were colored $\phi(F)$. 
Then $F$ together with other $r$ of them form a monochromatic weak copy of $V_r$ with respect to $\phi$.

Assume next that all sets of size at most $\min\{a:\binom{a+f_k(2r-1)}{\le a}> r\}$ are colored $k+1$. Then the empty set and $r$ other such sets form a monochromatic weak copy of $V_r$. This proves \textbf{(i)}.

To prove \textbf{(ii)}, let us consider a consecutive level $k$-coloring $\psi:B_{g_k(r)-1}\rightarrow [k]$ defined by the positive integers $i_1,i_2,\dots, i_k$ such that $\psi$ does not admit a monochromatic weak copy of $V_r$. We ``add $\max\{a:\binom{a+g_k(r)}{\le a}\le r\}+1$ extra levels", i.e. we let $j_1=\max\{a:\binom{a+g_k(r)}{\le a}\le r\}+1$, and $j_{h+1}=i_h$ for all $1\le h \le k$ and set $N':=\left(\sum_{h=1}^{k+1}j_h\right)-1$. We claim that the corresponding consecutive level $(k+1)$-coloring $\psi'$ does not admit a monochromatic weak copy of $V_r$, which proves \textbf{(ii)}. Indeed, by definition the union of the first $j_1$ layers does not contain $r+1$ sets, so no monochromatic $V_r$ exists in this color. To see the $V_r$-free property of the other color classes, observe that for any set $F$ of size $j_1$, the cube $B_{F,[N']}$ has dimension $g_k(r)-1$, and the consecutive level  $k$-coloring that we obtain by restricting $\psi'$ to $B_{F,[N']}$ is isomorphic to $\psi$. If $G$ is the set corresponding to the bottom element of a copy $C$ of $V_r$, then for a $j_1$-subset $F$ of $G$, the copy $C$ belongs to $B_{F,[N']}$, so it cannot be monochromatic.
\end{proof}
To prove the theorem we proceed by induction on $k$. If one can use only one color, then all colorings are consecutive level $1$-colorings and $B_N$ does not admit a monochromatic $V_r$ if and only if $2^N\le r$, so $g_1(r)=f_1(r)=\lfloor \log r\rfloor +1$ and $c_1=1$.

Assume now that the statement of the theorem is proved for some $k
\ge 1$ and let us fix $\varepsilon>0$. Observe that using Claim  \ref{veeclaim} \textbf{(ii)} and the inductive hypothesis we obtain that for $r$ large enough we have 
\begin{equation}\label{gmax}
g_{k+1}(r)\ge g_k(r)+\max\left\{a:\binom{a+g_k(r)}{\le a}\le r\right\}+1,
\end{equation}
and $(c_k-\varepsilon)\log r\le g_k(r)\le (c_k+\varepsilon)\log r$. We claim that if $d_k$ is the constant that satisfies $(d_k+c_k)h(\frac{d_k}{d_k+c_k})=1$, then the maximum $a$ in Inequality (\ref{gmax}) is at least $(d_k-\varepsilon)\log r$. Indeed, there exist positive constants $c_0$ and $\delta$ such that

$$\binom{(d_k-\varepsilon)\log r+g_k(r)}{\le (d_k-\varepsilon)\log r}  \le \binom{(d_k+c_k)\log r}{\le (d_k-\varepsilon)\log r}\le c_0\binom{(d_k+c_k)\log r}{ (d_k-\varepsilon)\log r}$$

$$\le  c_02^{h(\frac{d_k-\varepsilon}{d_k+c_k})(d_k+c_k)\log r}= c_0r^{h(\frac{d_k-\varepsilon}{d_k+c_k})(d_k+c_k)}\le c_0r^{1-\delta}<r
$$
holds, where for the second inequality we used $d_k<c_k$ and for the penultimate inequality we used that the entropy function is strictly increasing in $(0,1/2)$. Therefore, we have $g_{k+1}(r)\ge (c_k+d_k-2\varepsilon)\log r$.

On the other hand, according to Claim \ref{veeclaim} \textbf{(i)}, we have
\begin{equation}\label{fmax}
f_{k+1}(r)\le f_k(2r-1)+\min\left\{a:\binom{a+f_k(2r-1)}{\le a}> r\right\}.
\end{equation}
By the inductive hypothesis, for sufficiently large $r$ we have 
$$(c_k-\varepsilon)\log r\le f_k(r)\le f_k(2r-1)\le (c_k+\varepsilon)\log (2r-1)\le (c_k+2\varepsilon)\log r.$$ We claim that the minimum $a$ in Inequality (\ref{fmax}) is at most $(d_k+\varepsilon)\log r$. Indeed, for some positive $\delta'$ and large enough $r$ we have

$$
\binom{(d_k+\varepsilon)\log r+f_k(2r-1)}{\le (d_k+\varepsilon)\log r} \ge \binom{(d_k+c_k)\log r}{ (d_k+\varepsilon)\log r}\ge \frac{1}{\sqrt{\log r}}2^{h(\frac{d_k+\varepsilon}{d_k+c_k})(d_k+c_k)\log r}$$

$$= \frac{1}{\sqrt{\log r}}r^{h(\frac{d_k+\varepsilon}{d_k+c_k})(d_k+c_k)}\ge \frac{r^{1+\delta'}}{\sqrt{\log r}}>r.
$$

Therefore, we have $f_{k+1}(r)\le (c_k+d_k+3\varepsilon)\log r$ and consequently
\[
(c_k+d_k-2\varepsilon)\log r\le g_{k+1}(r)\le f_{k+1}(r)\le (c_k+d_k+3\varepsilon)\log r,
\]
showing $c_{k+1}=c_k+d_k$. Plugging back to the defining equation $(d_k+c_k)h(\frac{d_k}{d_k+c_k})=1$ we obtain $c_{k+1}h(\frac{c_{k+1}-c_k}{c_{k+1}})=1$ as claimed.
\end{proof}

Note that Cox and Steele \cite{CS} obtained general but not tight upper bounds on the Ramsey number  $R(V_{r_1},\dots, V_{r_{s}},\Lambda_{r_{s+1}},\dots,\Lambda_{r_t})$. Theorem \ref{weakvee} is an improvement on their result in case all target posets are the same.

\vskip 0.5truecm

\section{Strong copies}

The lower bounds in most of our theorems are obtained via trivial colorings where sets of the same size receive the same color. Let us introduce the following parameters: let $m(P)=\max\{m:B_m$ does not contain a weak copy of $P\}$ and $m^*(P)=\max\{m:B_m$ does not contain a strong copy of $P\}$. We say that $Q\subset B_n$ is \textit{thin} if $Q$ contains at most one set from each level. Also, let 
$r^*(P)=\max\{r: B_r$ does not contain a thin, strong copy of $P\}$. Note that the corresponding weak parameter $r(P)=\max\{r: B_r$ does not contain a thin, weak copy of $P\}$ trivially equals $|P|-2$ as $B_{|P|-1}$ contains a chain of length $|P|$ and thus a weak copy of $P$. 

\vspace{2mm}

In the next proposition we prove some lower bounds using non-trivial colorings. A poset $P$ is said to be \textit{connected} if for any pair $p,q\in P$ there exists a sequence $r_1,r_2,\dots,r_k$ such that $r_1=p,r_k=q$ and $r_i,r_{i+1}$ are comparable for any $i=1,2,\dots,k-1$.

\begin{proposition}\label{conneasy}
If $P$ is a connected poset with $|P|\ge 2$ and $Q$ is an arbitrary poset, then we have

\textbf{(i)} $RR(P,Q)>m(P)+|Q|-2$,

\textbf{(ii)} $RR^{*}(P,Q)> m^*(P)+|Q|-2$,

\textbf{(iii)} $RR^{*}(P,Q)> r^*(Q)$.
\end{proposition}

\begin{proof}
Set $N=m(P)+|Q|-2$, $N^*=m^*(P)+|Q|-2$ and $R=[|Q|-2]$. Consider the colorings $\phi:B_N\rightarrow \{0,1,\dots,|Q|-2\}$ and $\phi^*:B_{N^*}\rightarrow \{0,1,\dots,|Q|-2\}$ defined by $\phi(F)=|F\cap R|$ and $\phi^*(G)=|G\cap R|$. Observe that $\phi$ and $\phi^*$ do not admit a rainbow copy of $Q$ as only $|Q|-1$ colors are used.

By definition of $m(P)$, for any set $T\subseteq R$ the family $\cF_T=\{F\subseteq [N]: F\cap R=T\}$ cannot contain a weak copy of $P$. Thus a monochromatic weak copy of $P$ (admitted by $\phi$) must contain two sets $F, F'$ with $F\in \cF_T$ and $F'\in \cF_{T'}$ such that $|T|=|T'|$ and $T \not = T'$. As $P$ is connected, we can choose $F, F'$ to be comparable. However, since each $F\in \cF_T$ is incomparable to each $F'\in \cF_{T'}$ as $T$ is incomparable to $T'$, this is a contradiction. So the coloring $\phi$ does not admit a monochromatic weak copy of $P$. This proves \textbf{(i)}, and one can prove \textbf{(ii)} in a similar way.


To see \textbf{(iii)} let us consider the trivial coloring $\phi:B_{r^*(Q)}\rightarrow \{0,1,\dots,r^*(Q)\}$ defined by $\phi(F)=|F|$. As $P$ is connected with $|P|\ge 2$, $\phi$ does not admit a monochromatic copy of $P$ and by definition of $r^*(Q)$, $\phi$ does not admit a rainbow strong copy of $Q$.
\end{proof}


\begin{proposition}\label{thin} If $n\ge 4$, then $r^*(A_n)=n+1$ holds.
\end{proposition}

\begin{proof} Let $\cF\subset B_n$ be a thin antichain. Then we claim $|\cF|\le n-2$ holds, which shows $r^*(A_n)\ge n+1$. Indeed, if $\emptyset\in\cF$ or $[n]\in \cF$, then $\cF=\{\emptyset\}$ or $\cF=\{[n]\}$. Also, if both a $1$-element and an $(n-1)$-element sets are in $\cF$, they have to be complements, and then no other sets can be in $\cF$.

For the upper bound we prove the stronger statement that $B_n$ contains a thin antichain of size $n-2$ with $|\cF\cap \binom{[n]}{n-1}|=0$.  We proceed by induction on $n$. The statement is trivial for $n=4$ and $n=5$. Let us assume the statement holds for some $n\ge 4$, and prove it for $n+2$. Hence we can find a thin antichain $\cF$ in $B_n$ that has cardinality $n-2$ and does not contain a set of size $n-1$. Then let $\cF'=\{F\cup \{n+1\}: F\in \cF\}\cup \{[n],\{n+2\}\}$. It is easy to see that $\cF'\subset B_{n+2}$ is a thin antichain of size $n$ without an $(n+1)$-element set.
\end{proof}

Proposition \ref{conneasy} and Proposition \ref{thin} together yield $RR^{*}(C_2,A_k)\ge k+2$, which is larger than both $e^*(C_2)(|A_k|-1)+f(A_k)=k+1$ and $R^{*}_{k-1}(C_2)=k-1$, showing that $C_2$ does not possess the property of Question \ref{firstQ} and that there exists a pair of posets for which Proposition \ref{easy} \textbf{(ii)} holds with a strict inequality.

\begin{defn}We say that the families $\cF_1,\cF_2,\dots, \cF_l$ are \textit{mutually comparable} if for any $F_i\in \cF_i$ and $F_j\in \cF$ with $1 \le i< j \le l$ we have $F_i \subseteq F_j$ or $F_j \subseteq F_i$, and they are \textit{mutually incomparable} if for any $F_i\in \cF_i$ and $F_j\in \cF$ with $1 \le i< j \le l$ we have $F_i \not\subseteq F_j$ and $F_j \not\subseteq F_i$.
\end{defn}

\begin{proof}[Proof of Theorem \ref{genstrong}]
Let us write $N=\lfloor\lambda^*_{max}(P)(k-1)\rfloor+m_k$ and let us consider a coloring $\phi:B_N\rightarrow \mathbb{N}$. Observe that if $\phi$ does not admit a monochromatic induced copy of $P$, then for any set $S\subseteq [m_k]$, $\phi$ must admit at least $k$ colors on the family $\cQ_S=\{S\cup T: T\subseteq [N]\backslash[m_k]\}$. Indeed, if there are at most $k-1$ 
colors on some $\cQ_S$, then consider the corresponding coloring $\phi'$ of $B_{N-m_k}$ such 
that $\phi'(\{i_1,i_2,\ldots, i_\ell\})=\phi(S\cup\{i_1+m_k,i_2+m_k,\ldots, i_\ell+m_k\})$ 
for every set $\{i_1,i_2,\ldots, i_\ell\}$ in $ B_{[N-m_k]}$.
Then $\phi'$ is a $(k-1)$-coloring of $B_{N-m_k}$, and one of the color classes has Lubell-mass strictly larger than $\lambda^*_{max}(P)$. So $\phi'$ admits a monochromatic induced copy of $P$ in $B_{N-m_k}$. This implies that $\phi$ admits a monochromatic induced copy of $P$ in $\cQ_S$.

By the definition of $m_k$, we can pick $k$ subsets $S_1,S_2,\dots,S_k$ of $[m_k]$ of size $\lfloor m_k/2\rfloor$. As the $S_i$'s form an antichain, the families $\cQ_{S_1},\cQ_{S_2},\dots,\cQ_{S_k}$ are mutually incomparable. By the above paragraph, on each of these families $\phi$ admits at least $k$ colors otherwise we find a monochromatic induced copy of $P$. But then we can pick a rainbow antichain from the $\cQ_{S_i}$'s greedily: a set $F_1$ from $\cQ_{S_1}$, then $F_2$ from $\cQ_{S_2}$  and so on with $\phi(F_i)\neq \phi(F_j)$ for all $i<j$. This completes the proof of the first part of Theorem \ref{genstrong}.

Now we prove the second part. For any $P$ other than $C_1$ or $C_2$, $\cF=\{\emptyset, [n]\}\subset B_n$ is $P$-free for all $n\ge2$. Hence $\lambda_{max}^*(P)=\sup \lambda_n^*(P)\ge 2$.
Let $N=\lfloor2\lambda^*_{max}(P)\rfloor+2$. 
For any coloring $\psi$ of $B_N^-$, we show that it admits either  a monochromatic induced copy of $P$ or a rainbow copy of $A_3$. 
If $\|\psi\|\le 2$, then $\lambda_N^*(B_N^-)=N-1$ hence one of the color classes has Lubell-mass strictly larger than $\lambda^*_{max}(P)$, so by the definition of $\lambda^*_{max}$, $\psi$ admits a monochromatic induced copy of $P$.

Therefore, we can assume that $\|\psi\|\ge 3$. Let $\cQ_{i}=\{\{i\}\cup T: T\subseteq [N]\setminus [2]\}$  for $i=1,2$. Note that $\cQ_{1}$ and $\cQ_{2}$ are mutually incomparable. 
By the same reasoning as the previous case, if $\psi$ admits only 2 colors on some  $\cQ_{i}$, then we can find a corresponding 2-coloring $\psi'$ of $B_{N-2}$ and a monochromatic copy of $P$ in $B_{N-2}$  with respect to $\psi'$. 
As before, this implies that there is a monochromatic copy of $P$ in $\cQ_{i}$ with respect to $\psi$. Hence we consider the case that $\psi$ admits at least three colors on each $\cQ_{i}$.  
If there are two sets $F_1,F_2\in \cQ_{1}$ of the same size with distinct colors, then a set of third color in $\cQ_{2}$ together with $F_1$ and $F_2$ form a rainbow $A_3$. So we may assume that all subsets of the same size in $\cQ_{1}$ have the same color. Now if all sets in  $\cQ_{1}\setminus\{\{1\},([N]\setminus[2])\cup\{1\}\}$ are of the same color, then the  corresponding coloring $\psi'$ admits only one color on   $B_{N-2}^-$. 
Since $\lambda^*_{max}(P)\ge 2$, we have $\lambda_{N-2}^*(B_{N-2}^-)=
N-3=\lfloor2\lambda^*_{max}(P)\rfloor-1> \lambda^*_{max}(P)$.
Thus, $\psi'$ admits a monochromatic $P$ in $B_{N-2}$ 
and then $\psi$ admits a monochromatic $P$ in $\cQ_{1}$ as well.
If there are at least two colors on $\cQ_{1}\setminus\{\{1\},([N]\setminus[2])\cup\{1\}\}$
and sets of the same size have the same color, then we can easily find two incomparable sets from two levels of distinct colors. The two sets together with a set of third color in $\cQ_{2}$ form a rainbow $A_3$. This completes the proof.  
\end{proof}

\section{$F(n,k)$, $F'(n,k)$, $G(n,k)$ and $G'(n,k)$}

Most of our proofs proceeded along the following lines: suppose $\phi$ is a coloring using at least $c\ge k$ colors that does not admit a rainbow $A_k$, then the union of some $c-k+1$ color classes is ``small''. Therefore it is natural to investigate the following four functions that seem to be interesting in their own right.

\begin{defn}
$F(n,k)$ is the smallest integer $m$ such that any $k$-coloring $\phi:B_n\rightarrow [k]$ admits a strong rainbow copy of $A_k$ provided every color class is of size at least $m$. $G(n,k)$ is the infimum of all reals $\gamma$ such that any $k$-coloring $\phi:B_n\rightarrow [k]$ admits a strong rainbow copy of $A_k$ provided every color class has Lubell-mass at least $\gamma$.

$F'(n,k)$ and $G'(n,k)$ are defined by changing coloring to \textit{partial coloring} (i.e. we only color some subset of the elements of $B_n$) in the definition of $F(n,k)$ and $G(n,k)$.
\end{defn}

We start with calculating the Lubell mass of subcubes of $B_n$.

\begin{lemma}\label{cubes2}
For non-negative integers $a$ and $b$ with $a+b\le n$, we have
$$\lambda_n(B_{a,n-b})=\frac{n+1}{a+b+1}\frac{1}{\binom{a+b}{a}}.$$
\end{lemma}

\begin{proof}
$$\lambda_n(B_{a,n-b})=\sum_{i=a}^{n-b} \frac{\binom{n-a-b}{i-a}}{\binom{n}{i}}=
\frac{(n-a-b)!}{n!}\sum_{i=a}^{n-b} \frac{i!(n-i)!}{(i-a)!(n-b-i)!}
$$
$$=\frac{(n-a-b)!a!b!}{n!}\sum_{i=a}^{n-b} \binom{i}{a} \binom{n-i}{b}=\frac{(n-a-b)!a!b!}{n!} \binom{n+1}{a+b+1}$$
$$=\frac{(n-a-b)!a!b!(n+1)!}{n!(a+b+1)!(n-a-b)!}=\frac{n+1}{a+b+1}\frac{1}{\binom{a+b}{a}}.$$

Here we use the equation $\sum_{i=a}^{n-b} \binom{i}{a} \binom{n-i}{b}=\binom{n+1}{a+b+1}$. This can be proved the following way. The right hand side denotes the number of ways to pick an $(a+b+1)$-element subset $\{x_1,\dots,x_{a+b+1}\}$ of $[n+1]$ with $x_1<x_2<\dots <x_{a+b+1}$. Let us assume $x_{a+1}=i+1$. Then $i$ is between $a$ and $n-b$, there are $\binom{i}{a}$ ways to pick $\{x_1,\dots,x_{a}\}$ and $\binom{n-i}{b}$ ways to pick $\{x_{a+2},\dots,x_{a+b+1}\}$.
\end{proof}

Let us start with the following simple observation that connects the four functions.

\begin{proposition}\label{fourf}
For any $n$ and $k$ we have $F(n,k-1)\le F'(n,k-1)$ and $G(n,k-1)\le G'(n,k-1)$. Furthermore if $F'(n,k-1)\le \frac{2^n}{k}+1$, then $F'(n,k-1)\le F(n,k)$.
\end{proposition}

\begin{proof}
The inequalities $F(n,k-1)\le F'(n,k-1)$ and $G(n,k-1)\le G'(n,k-1)$ are immediate as any coloring is a special partial coloring. 
Let $m=F'(n,k-1)$ and let $\phi'$ be a partial $(k-1)$-coloring of $B_n$ such that $\phi'$ has no rainbow $A_{k-1}$ and every color class in $\phi'$ has size at least $m-1$.  We may assume without loss of generality that each color class of $\phi'$ has size exactly $m-1$.  Let $\phi$ be the $k$-coloring of $B_n$ obtained from $\phi'$ by assigning color $k$ to each uncolored set. As $\phi'$ does not admit a rainbow $A_{k-1}$, $\phi$ does not admit a rainbow $A_k$. Since there are $2^n - (k-1)(m-1)$ sets of color $k$ under $\phi$ and $2^n - (k-1)(m-1) \ge (m-1)$ when $m\le \frac{2^n}{k} + 1$, it follows that $F(n,k)>m-1$.
\end{proof}





\begin{defn}Let $\cF_1,\cF_2,\dots, \cF_l$ be pairwise disjoint families in $B_n$. A chain $\emptyset=S_0\subseteq S_1\subseteq S_2 \subseteq \dots \subseteq S_m= [n]$ is called a \textit{core chain} if $\bigcup_{i=1}^l \cF_i\subseteq \bigcup_{j=0}^{m-1} B_{S_{j}S_{j+1}}$ and none of the truncated subcubes $B^-_{S_{j}S_{j+1}}$ contain elements from two different families $\cF_{i_1}$ and $\cF_{i_2}$. \end{defn}

The next simple lemma is more or less due to Ahlswede and Zhang \cite{AZ}, we include the proof for completeness.

\begin{lemma}\label{mcs}
If $\cF_1,\cF_2,\dots, \cF_l\subseteq B_M$ are mutually comparable, then they have a core chain.

\end{lemma}

\begin{proof}
We proceed by induction on $M$, with the base case $M=1$ being trivial. Suppose the statement has been proved for any $M'<M$, let $\cF_1,\cF_2,\dots, \cF_l\subseteq B_M$ be mutually comparable families of sets, and let $\cF:=\bigcup_{i=1}^l \cF_i$. Let $F$ be a set of maximum size in $\cF \setminus \{[M]\}$  and suppose $F\in \cF_i$. Then consider the family $\cF'_i$ of those sets in $\cF_i$ that are not contained in any $H\in \cF\setminus \cF_i$ with $H\neq [M]$. By its maximum size, $F$ belongs to $\cF'_i$. Set $S:=\cap_{F'\in \cF'_i}F'$. The sets of $\cF'_i$ obviously contain $S$. The sets of $\cF\setminus \cF_i$ are contained in $S$ by the definition of $\cF_i'$ and the mutually comparable property. Finally, the sets of $\cF_i\setminus \cF'_i$ are contained in $S$ because they are contained in some  $H\in \cF\setminus \cF_i$. 
We apply induction to $\cF_1,\dots, \cF_{i-1},\cF_i\setminus \cF'_i,\cF_{i+1},\dots, \cF_l\subseteq B_S$. Adding $[M]$ to the  resulting core chain will give us the desired core chain. 
\end{proof}

\begin{corollary}\label{2} We have:

\smallskip 
\textbf{(i)} $F'(n,2)=2^{\lfloor n/2\rfloor}+2$ if $n\ge 5$ is odd, and $F'(n,2)=2^{\lfloor n/2\rfloor}$ if $n$ is even.

\smallskip
\textbf{(ii)}
If $\cF_1,\cF_2\subseteq 2^{[n]}$ are mutually comparable families with a core chain that has neither a set of size $\lfloor n/2\rfloor$ nor a set of size $\lceil n/2\rceil$,
then $$\min\{|\cF_1|,|\cF_2|\}\le 2^{\lfloor n/2\rfloor-1}+4.$$

\smallskip
\textbf{(iii)} $\lim_{n\rightarrow \infty} G'(n,2)=\sqrt{2}+1$. 
\end{corollary}

\begin{proof}
First we prove \textbf{(i)}. For the lower bound consider a set $S$ of size $\lfloor n/2\rfloor$ and for any $T\subseteq S$ let us define $\phi(T)=1$. If $n$ is even, then let $\phi(T')=2$ for any $T'\in \cU_S\setminus \{S\}$. Then $\phi$ does not admit a rainbow $A_2$, $|\phi^{-1}(1)|=2^{n/2}$ and $|\phi^{-1}(2)|=2^{n/2}-1$ so $F'(n,2)>2^{n/2}-1$. If $n$ is odd, then let $\phi([n])=1$ and $\phi(T)=2$ for any set $T\in B^-_{S,[n]}$. Again, $\phi$ does not admit a rainbow $A_2$, $|\phi^{-1}(1)|=2^{\lfloor n/2\rfloor}+1$ and $|\phi^{-1}(2)|=2^{\lceil n/2\rceil}-2$, so if $n\ge 5$, then $F'(n,2)>2^{\lfloor n/2\rfloor}+1$ for odd values of $n$.

To obtain the upper bound observe that a partial coloring $\phi:B_n\rightarrow \{1,2\}$ does not admit a rainbow $A_2$ if and only if the color classes $\phi^{-1}(1)$ and $\phi^{-1}(2)$ are mutually comparable. Therefore applying Lemma \ref{mcs} to $\phi^{-1}(1)$ and $\phi^{-1}(2)$ we obtain a core chain $\emptyset=S_0\subset S_1 \subset \dots \subset S_l\subset S_{l+1}=[n]$ with $\cup_{j=0}^{l}B_{S_j,S_{j+1}}$ containing both color classes such that for every $j$ the truncated subcube $B_{S_j,S_{j+1}}^-$ contains sets only from one color class. Let $d_j=|S_{j+1} \setminus S_j|$ and thus $\sum_{j=0}^ld_j=n$. Then the size of color class 1 is at most $2^{d_{j_1}}+2^{d_{j_2}}+\ldots + 2^{d_{j_i}}+k-1$ where the $j_h$'s are the indices of  the subcubes containing only sets of color 1 and $k$ is the number of $S_j$'s of color 1 with both $B^-_{S_{j-1},S_j}$ and $B^-_{S_j,S_{j+1}}$ containing only sets of color 2. As for positive integers $x,y$ we have $2^x+2^y\le 2^{x+y}-2$ (unless $x=y=1$), to maximize the minimum size of the color classes, we must have  only two subcubes in the partition. A simple case analysis based on the size of $S_1$ and the colors of $\emptyset,S_1$ and $[n]$ finishes the proof of \textbf{(i)}.

\smallskip 

To prove \textbf{(ii)}, let $\emptyset=S_1,S_2,\dots,S_j=[n]$ be the core chain 
and let us write $d_h=|S_h\setminus S_{h-1}|$ for any $1\le h\le j$. Let us give all the details of this case analysis. If there exists a $d_j>\lceil n/2\rceil$, then the sum of the $d_i$'s corresponding to the other color class is at most $\lfloor n/2\rfloor -1$, so the color class has size at most $2^{\lfloor n/2\rfloor-1}+2$. As, by the assumption of \textbf{(ii)}, no $S_i$ has size $\lfloor n/2\rfloor$ or $\lceil n/2\rceil$, there can be at most one $d_j$ with $\lfloor n/2\rfloor \le d_j\le \lceil n/2\rceil$. If there is such a $d_j$, then  the $d_i$'s belonging to the other color sum up to at most $\lfloor n/2\rfloor+2$ and all of them are at most $\lfloor n/2\rfloor -1$. This yields that the size of this color class is at most $2^{\lfloor n/2\rfloor-1}+4$. (Note that this is sharp if $n$ is odd, $d_1=\lfloor n/2\rfloor -1$, $d_2=\lfloor n/2\rfloor$, $d_3=2$.) Finally, suppose that all $d_j$'s are at most $\lfloor n/2\rfloor -1$. Note that there are at most 2 $d_j$'s larger than $n/3$. If all $d_i$'s belonging to one color class are at most $\lfloor n/2\rfloor -2$, then one of the color classes has size at most $2^{\lfloor n/2\rfloor-1}$. If both color classes have a $d_i$ that equals $\lfloor n/2\rfloor -1$, then the remaining $d_i$'s sum to 2 or 3 (depending on the parity of $n$), so one of the color classes have size at most $2^{\lfloor n/2\rfloor -1}+3$.

\smallskip

To prove \textbf{(iii)}, first we show that $\liminf G'(n,2)\ge \sqrt{2}+1$. Let us fix a  set $H$ of size $\lfloor \frac{n}{\sqrt{2}}\rfloor$. Let us color $\cU_H \cup \{\emptyset\}$ by 1 and let us color  $\cD_{H}\setminus \{\emptyset\}$ by 2. Applying Lemma \ref{cubes2}, we obtain that the Lubell mass of  both color classes is $\sqrt{2}+1+o(1)$. 

Now we prove $\limsup G'(n,2)\le \sqrt{2}+1$ through a sequence of claims. Throughout the proof we will assume that $n$ is sufficiently large. Let $\phi:B_n\rightarrow \{1,2\}$ be a partial coloring that does not admit a rainbow copy of $A_2$ and let us write $\cH_1=\phi^{-1}(1)$ and $\cH_2=\phi^{-1}(2)$. Finally, let $\emptyset=S_0,S_1,\dots, S_l,S_{l+1}=[n]$ be a core chain for this pair of mutually comparable families. 

\begin{claim}\label{g1}
 If $|S_h|\neq 0, |S_{h+1}|\neq n$ and $\lambda (B_{S_h,S_{h+1}})\ge n^{-1/3}$, then $\max\{|S_h|,n-|S_{h+1}|\}\le n^{2/3}$.
\end{claim}

\begin{proof}[Proof of Claim]
By Lemma \ref{cubes2}, we obtain $$\lambda_n(B_{1,n-\floor{n^{2/3}}})=\frac{n+1}{2+\floor{n^{2/3}}}\frac{1}{1+\floor{n^{2/3}}} < 
\frac{n+1}{(1+n^{2/3})n^{2/3}} < n^{-1/3}.$$
\end{proof}

\begin{claim}\label{g2}
If $\min \{\lambda_n(\cH_1),\lambda_n(\cH_2)\}\ge 2.1$, then $l\ge 1$ and either $|S_1|>n/2$ or $|S_l|<n/2$.
\end{claim}

\begin{proof}
If $l=0$, then one of the color classes is a subfamily of $\{\emptyset, [n]\}$ and thus its Lubell mass is at most 2. 

Now we proceed by contradiction. Suppose $|S_1|\le n/2$ and $|S_l|\ge n/2$; then there exists an index $1\le j \le l$ such that $|S_j|\le n/2, |S_{j+1}|>n/2$. Observe that $\lambda_n( \cD_{S_j})\le 2$ and $\lambda_n(\cU_{S_{j+1}})\le 2$ by Lemma \ref{cubes2}. Observe furthermore that $\cH_1\cup \cH_2 \subseteq \cD_{S_j} \cup B_{S_j,S_{j+1}}\cup \cU_{S_{j+1}}$ holds and one of the color classes is contained in $\cD_{S_j} \cup \cU_{S_{j+1}}$. 

\smallskip

If $\lambda_n(B_{S_j,S_{j+1}})\le n^{-1/3}$, then $\lambda_n(\cH_1\cup \cH_2)\le 4+n^{-1/3}$ holds and thus $\min \{\lambda_n(\cH_1),\lambda_n(\cH_2)\}< 2.1$.

\smallskip 

If $\lambda_n(B_{S_j,S_{j+1}})> n^{-1/3}$, then by Claim \ref{g1}, we have $\max\{|S_j|,n-|S_{j+1}|\}\ge n^{2/3}$ and thus $\lambda_n(\cD_{S_j} \cup \cU_{S_{j+1}})\le 2+o(1)$. Since $\cD_{S_j} \cup \cU_{S_{j+1}}$ contains $\cH_1$ or $\cH_2$, $\min \{\lambda_n(\cH_1),\lambda_n(\cH_2)\}< 2.1$. 
\end{proof}

By Claim \ref{g2}, we can assume without loss of generality that $l\ge 1$, $|S_1|>n/2$ and $B^-_{\emptyset,S_1}\cap \cH_2=\emptyset$. Then $\lambda_n(\cH_2)\le 1+ \lambda_n(B_{S_1,[n]})$ and in order to have $\lambda_n(\cH_2)\ge 2.1$, we must have $\emptyset,[n]\in \cH_2$. 

If $\cH_1$ contains a set $S$ of size at least $0.99n$, then for some $S'\subseteq S\subseteq S''$, we have $\cH_2\subseteq \{\emptyset\} \cup B_{S_1,S'}\cup B_{S'',[n]}$ and thus, by Lemma \ref{cubes2}, $\lambda_n(\cH_2)\le 1+\lambda_n(B_{|S_1|,n-1})+\lambda_n(B_{0.99n,n})\le 1+O(\frac{1}{n})+\frac{1}{0.99}<1+\sqrt{2}$.

Therefore, we can assume that $\cH_1\subseteq B_{\emptyset,S_1}\setminus \{\emptyset\}\cup B_{S_1,[n]}\cap \binom{[n]}{\le 0.99n}$.
But as $|S_1|\ge n/2$, we have $$\lambda_n(B_{S_1,[n]}\cap \binom{[n]}{\le 0.99n})=\sum_{i=0.01n}^{n-|S_1|}\frac{\binom{n-|S_1|}{i}}{\binom{n}{i}}\le \frac{n}{2}\frac{\binom{n/2}{0.01n}}{\binom{n}{0.01n}}=o(1).$$  
 Therefore $\lambda_n(\cH_1)\le \lambda_n(B_{0,|S_1|})-1+o(1)$. It follows that 
$$\limsup_n G'(n,2)\le \limsup_n\max_{n/2<|S_1|\le n}\min\{1+ \lambda_n(B_{S_1,[n]}), \lambda_n(B_{0,|S_1|})-1\}.$$
Writing $|S_1|=cn$ for some $1/2\le c\le 1$ and applying Lemma \ref{cubes2}, we obtain that  the right hand side is $\max\min\{1+\frac{1}{c},\frac{1}{1-c}-1\}$. The first function is decreasing in $c$, the second one is increasing, so the minimum is maximized when the two functions are equal. This happens at $c=\frac{1}{\sqrt{2}}$ and thus $\limsup G'(n,2)\le \sqrt{2}+1$.
\end{proof}

\begin{theorem}\label{23}
\textbf{(i)} For $n\ge 18$ we have $F(n,3)=F'(n,2)$.

\smallskip 

\textbf{(ii)} We have $\lim_{n\rightarrow \infty}G(n,3)=\lim_{n\rightarrow \infty}G'(n,2)=1+\sqrt{2}$.
\end{theorem}

\begin{proof}
Let $\phi:B_n\rightarrow [3]$ be a coloring that does not admit any strong rainbow $A_3$. Let $\cH_1,\cH_2,\cH_3$ denote the three color classes. As every level admits at most two colors, we can assume without loss of generality that $\cH_1$ contains at least $\frac{\binom{n}{\lfloor n/2\rfloor}}{2}$ sets of size $\lfloor n/2\rfloor$.

\begin{claim}\label{23claim}
Suppose either $|\cH_2|,|\cH_3|> n(n+1)+2=\binom{n}{0}+\binom{n}{1}+\binom{n}{2}+\binom{n}{n-2}+\binom{n}{n-1}+\binom{n}{n}$, or $\lambda_n(\cH_2),\lambda_n(\cH_3)\ge 2.1$ and $n$ is large enough. Then $\cH_2$ and $\cH_3$ are mutually comparable possibly with the exception of some complement set pairs of size $1$ and $n-1$.
\end{claim}

\begin{proof}[Proof of Claim]
First suppose there exists an incomparable pair $H_2\in \cH_2, H_3\in \cH_3$ with $3\le |H_2|,|H_3|\le n-3$. Then the number of sets of size $\lfloor n/2\rfloor$ that are comparable either to $H_2$ or to $H_3$ is at most $2\binom{n-3}{\lfloor n/2\rfloor}<\frac{\binom{n}{\lfloor n/2\rfloor}}{2}$. Therefore, there exists a set $H_1\in \cH_1$ of size $\lfloor n/2\rfloor$ such that $H_1,H_2,H_3$ form an antichain. This contradicts that $\phi$ does not admit a strong rainbow copy of $A_3$.

By the above, we can take $\cH_2'\subseteq \cH_2, \cH'_3\subseteq \cH_3$  maximal mutually comparable subfamilies such that they contain all sets from $\cH_2$ and $\cH_3$ of size between $3$ and $n-3$. We claim that, by any of the two possible assumptions on $|\cH_2|$ and $|\cH_3|$, there exist $H_2\in \cH_2, H_3\in \cH_3$ with $3\le |H_2|,|H_3|\le n-3$ and thus by definition of $\cH_2'$ and $\cH_3'$, we have that $H_2\in \cH_2', H_3\in \cH_3'$. This is clear for the case $|\cH_2|,|\cH_3|> n(n+1)+2=\binom{n}{0}+\binom{n}{1}+\binom{n}{2}+\binom{n}{n-2}+\binom{n}{n-1}+\binom{n}{n}$. 
Suppose next $\lambda_n(\cH_2),\lambda_n(\cH_3)\ge 2.1$ holds and assume that $\cH_3\subseteq \{\emptyset,[n]\} \cup \binom{[n]}{1} \cup \binom{[n]}{2} \cup \binom{[n]}{n-2} \cup \binom{[n]}{n-1}$. We will first show that $\cH_2\cap \binom{[n]}{\lfloor n/2\rfloor}=o\left(\binom{n}{\lfloor n/2\rfloor}\right)$.

As $\lambda(\cH_3)\ge 2.1$, there exists $\varepsilon>0$ and $i\in \{1,2,n-2,n-1\}$ such that $|\cH_3\cap \binom{[n]}{i}|\ge \varepsilon\binom{n}{i}$. If $i=1$, then all sets $H\in \cH_2\cap \binom{[n]}{\lfloor n/2\rfloor}$ must contain all but at most one singleton of $\cH_3\cap \binom{[n]}{1}$. Indeed if $x,y\notin H$ with $\{x\},\{y\}\in \cH_3$, then the number of those $\lfloor n/2\rfloor$-sets that contain at most one of $x,y$ is $(\frac{3}{4}+o(1))\binom{n}{\lfloor n/2\rfloor}$, so one of these, say $H'$, is colored 1 and then $H,H'$ and one of $\{x\}$ and $\{y\}$ would form an induced rainbow copy of $A_3$. The number of $\lfloor n/2\rfloor$-sets that contain all but one of $\varepsilon n$ singletons is not more than $\varepsilon n\binom{n-\varepsilon n+1}{\lfloor n/2\rfloor -\varepsilon n+1}=o\left(\binom{n}{\lfloor n/2\rfloor}\right)$. 

If $i=2$, then we claim that every set $H\in \cH_2\cap \binom{[n]}{\lfloor n/2\rfloor}$ must contain every pair $P\in \cH_3\cap \binom{[n]}{2}$. Indeed, as $|\cH_1\cap \binom{[n]}{\lfloor n/2\rfloor}|\ge \frac{1}{2}\binom{n}{\lfloor n/2\rfloor}$, there is an $H'\in \cH_1\cap \binom{[n]}{\lfloor n/2\rfloor}$ not containing $P$ so $P,H,H'$ would form an induced  rainbow copy of $A_3$. Observe that $|\cH_3\cap \binom{[n]}{2}|\ge \varepsilon \binom{n}{2}$ implies $|\cup_{P\in \cH_3\cap \binom{[n]}{2}}P|\ge \varepsilon' n$. Then similarly to the case $i=1$, we have  $\cH_2\cap \binom{[n]}{\lfloor n/2\rfloor}=o\left(\binom{n}{\lfloor n/2\rfloor}\right)$, since the number of $\lfloor n/2\rfloor$-sets containing a fixed set of size $\varepsilon'n$ is $o(\binom{n}{\lfloor n/2\rfloor}$. The cases $i=n-2,n-1$ are analogous to the cases $i=2,1$, so we can assume that $(1-o(1))\binom{n}{\lfloor n/2\rfloor}$ sets in $\binom{[n]}{\lfloor n/2\rfloor}$ belong to $\cH_1$. 

Let $H'\in \cH_2$ and $H''\in \cH_3$ be incomparable sets that are not complements of size 1 and $n-1$ (if there are no such sets, then we are done with the proof of the claim). Then the number of $\lfloor n/2\rfloor$-sets that are comparable to at least one of $H'$ or $H''$ is at most $(3/4+o(1))\binom{n}{\lfloor n/2\rfloor}$. Therefore, there exists $H\in \cH_1$ that forms an induced rainbow $A_3$ with $H'$ and $H''$. This contradiction shows that there must exist $H_2\in \cH_2$ and $H_3\in \cH_3$ with $3\le H_2,H_3\le n-3$.

The existence of $H_2,H_3$ shows that Lemma \ref{mcs} applied to $\cH_2'$ and $\cH_3'$ must yield a core chain $S_1,S_2,\dots, S_j$ that contains a set $S_i$ with $3\le |S_i|\le n-3$. As all sets of size $\lfloor n/2\rfloor$ that are colored 2 or 3 must be comparable to $S_i$, we obtain that the number of sets of size $\lfloor n/2\rfloor$ colored 1 is at least $\binom{n}{\lfloor n/2\rfloor}-\binom{n-3}{\lfloor n/2\rfloor}$. 

Suppose finally that there exist incomparable sets $H \in \cH_2, H'\in \cH_3$ such that $H,H'$ are not complement pairs of size 1 and $n-1$. We claim that there must exist an $\lfloor n/2\rfloor$-set $H''$ colored 1 that is incomparable to both $H$ and $H'$ contradicting the fact that $\phi$ does not admit a strong rainbow $A_3$. Indeed, the two worst cases are the following.

$\bullet$ If $|H|=|H'|=1$ or $|H|=|H'|=n-1$, then the number of $\lfloor n/2 \rfloor$-sets that are  comparable  to at least one of $H$ and $H'$ is $\binom{n}{\lfloor n/2\rfloor}-\binom{n-2}{\lfloor n/2\rfloor}<\binom{n}{\lfloor n/2\rfloor}-\binom{n-3}{\lfloor n/2\rfloor}$, so indeed, there exists at least one $H''$ incomparable to both $H$ and $H'$.

$\bullet$
If $|H|=1,|H'|=n-2$, then the number of $\lfloor n/2 \rfloor$-sets that are  comparable to $H$ or $H'$ is $\binom{n-2}{\lfloor n/2\rfloor}+\binom{n-1}{\lfloor n/2\rfloor-1}<\binom{n}{\lfloor n/2\rfloor}-\binom{n-3}{\lfloor n/2\rfloor}$, so again, there exists at least one $H''$ incomparable to both $H$ and $H'$.
\end{proof}

With Claim \ref{23claim} in hand, we are ready to prove \textbf{(i)}. To prove $F(n,3)\le F'(n,2)$ we need to show that one of the color classes has size smaller than $F'(n,2)$. If $n\ge 18$, then $n(n+1)+2< F'(n,2)$, so if a color class of $\phi$ has size at most $n(n+1)+2$, then we are done. If two color classes of $\phi$ are mutually comparable, then by definition, one of these has size less than $F'(n,2)$ and we are done again. By Claim \ref{23claim}, the only other possibility is that $\cH_2,\cH_3$ contains a pair of complement sets $\{x\}\in \cH_2,[n]\setminus \{x\}\in \cH_3$, and there is a maximal mutually comparable pair of subfamilies $\cH'_2,\cH'_3$ such that $\cH_2'$ and $\cH_3'$ contain all sets from $\cH_2$ and $\cH_3$ of size between 2 and $n-2$. Consider the core chain $S_0,S_1,\dots, S_j$ 
of $\cH'_2$ and $\cH'_3$. 

Let us assume first that no  $S_i$ has size $\lfloor n/2\rfloor$ or $\lceil n/2\rceil$. Then by \textbf{(ii)} of Corollary \ref{2}, either $\cH'_2$ or $\cH'_3$ has size at most $2^{\lfloor n/2\rfloor-1}+4$ and thus one of $\cH_2,\cH_3$ has size at most $2^{\lfloor n/2\rfloor-1}+2n+4< 2^{\lfloor n/2\rfloor}\le F'(n,2)$, where the first inequality holds by the assumption $n\ge 18$. 

Finally, suppose $|S_i|=\lfloor n/2\rfloor$. We know that either $x$ or $[n]\setminus \{x\}$ is not comparable to $S_i$. 

\bigskip

\textsc{Case I}
$x\notin S_i$. 

As all sets of $\cH_3$ but $[n]\setminus \{x\}$ are comparable to $\{x\}$, and all sets of $\cH_3$ of size between 1 and $n-2$ are comparable to $S_i$, we must have $\cH_3\subseteq \cU_{S_i\cup \{x\}} \cup \binom{[n]}{n-1}\cup\{\emptyset\}$. If $n$ is even, then this means $|\cH_3|\le 2^{n/2-1}+n/2+2<2^{n/2}=F'(n,2)$  ($n\ge 18$ was used for the second inequality) and we are done. 

So we can suppose $n=2m+1$. If $\cH_2$ contains a set $F\neq S_i\cup \{x\}$ of size between $m+1$ and $n-1$, then $\cH_3\subseteq B_{S_1\cup\{x\},F\cup \{x\}}\cup  \cU_{F\cup \{x\}} \cup \binom{[n]}{n-1}\cup \{\emptyset\}$ so $|\cH_3|\le 2^{m-1}+n<2^m+2=F'(n,2)$ and we are done again. So we may assume $\cH_2 \subseteq \cD_{S_i}\cup \{S_i\cup\{x\},[n]\}\cup\binom{[n]}{1}$. First we claim that $\cH_2$ cannot contain $\{y\}$ with $x\neq y\notin S_i$.  Indeed, by Claim \ref{23claim}, it would yield that $\cH_3\subseteq \cU_{S_i\cup\{x,y\}}\cup \binom{[n]}{n-1}\cup \{\emptyset\}$ and thus $\cH_3$ is of size at most $2^{m-1}+n+1< F'(n,2)$. This shows $\cH_2 \subseteq \cD_{S_i}\cup \{S_i\cup\{x\},\{x\},[n]\}$. Next, we claim that either $|\cH_2|<2^m+2$ or $\cH_3 \subseteq \cU_{S_i\cup\{x\}} \cup \{\emptyset\}$. Indeed,  if $\cH_3$ contains a set $[n]\setminus \{y\}$ with $y\in S_i$, then $\cH_2$ cannot contain sets from $B^-_{\{y\},S_i}$, so its size is at most $2^{m-1}+3$. If  $|\cH_2|<2^m+2=F'(n,2)$, then we are done. Otherwise $|\cH_3|\le |\cU_{S_i\cup\{x\}} \cup \{\emptyset\}|=2^m+1<F'(n,2)$ and we are done.

\bigskip
\textsc{Case II} for any pair $x\in\cH_2,[n]\setminus \{x\}\in\cH_3$ we have $S_i\not\subseteq [n]\setminus \{x\}$. 

If $n$ is even then the coloring defined by $\phi'(F):=\phi([n]\setminus F)$ has color classes of the same size as those of $\phi$, and $S_i$ is replaced by $[n]\setminus S_i$, a set also of size $n/2$ and then we are back to the previous case. Thus we can assume that $n=2m+1$ is odd. Then as all sets of $\cH_2$ of size between 2 and $n$ are comparable to $[n]\setminus \{x\}$ and all sets of $\cH_2$ are comparable to $S_i$, we obtain that $\cH_2\subseteq \cD_{S_i\setminus \{x\}} \cup \{[n]\}$, so $|\cH_2|\le 2^{m-1}+1<F'(n,2)$ and we are done.

The inequality $F'(n,2)\le F(n,3)$ follows from Proposition \ref{fourf} and Corollary \ref{2} \textbf{(i)}. This proves part $\textbf{(i)}$ of the theorem.

\medskip 

To show \textbf{(ii)}, we need to prove that for any coloring $\phi:B_n\rightarrow [3]$ that does not admit a rainbow induced copy of $A_3$, one of the color classes has Lubell mass at most $1+\sqrt{2}+o(1)$. Again, we assume that the color class $\cH_1$ contain at least half of $\binom{[n]}{\lfloor n/2\rfloor}$.

Claim \ref{23claim} 
implies that either one of the color classes of $\phi$ has Lubell mass at most $2.1\le 1+\sqrt{2}+o(1)=G'(n,2)+o(1)$ 
and we are done, or for any incomparable pair $H_2\in \cH_2, H_3\in \cH_3$ we have that $H_2$ is a singleton and $H_3$ is its complement. 
Let $R=\{x\in [n]: \{x\}\in\cH_2, \, [n]\setminus \{x\}\in \cH_3\}$ and we define $\beta$ by $|R|=\beta n$. Note that $R\neq \emptyset$, otherwise $\cH_2$ and $\cH_3$ are mutually comparable. Observe that if we let $\cH_3'=\cH_3\setminus \{[n]\setminus \{x\}:x\in R\}$, then $\cH_2$ and $\cH_3'$ are mutually comparable and $\lambda_n(\cH_3\setminus \cH_3')=\beta$. 

Let us consider the core chain
$\emptyset=S_0,S_1,\dots,S_j=[n]$ of $\cH_2$ and $\cH_3'$. As $S_1$ is comparable to each of the singletons in $\cH_2$, we know that $B^-_{\emptyset,S_1}$ is disjoint with $\cH_3'$ and if we define $\alpha$ by $|S_1|=\alpha n$, then we have $0\le \beta \le \alpha\le 1$. We claim that any $H\in \cH_2$ is a subset of $S_1\setminus R$ unless $H=[n]$ or $H=\{x\}$ for some $x\in R$. Indeed, if $H$ is none of the above members of $\cH_2$, then by
definition of the core chain, $H$ is comparable to $S_1$, thus we have either $S_1\subsetneq H$ or $x\in H\subseteq S_1$ for some $x\in R$. In both cases, $H$ contains at least one element $x$ of $R$. Then $H$ is incomparable to $[n]\setminus \{x\}$. But $[n]\setminus \{x\}\in \cH_3$ if $x\in R$, and, by Claim \ref{23claim}, $[n]\setminus \{x\}$ is incomparable only to one member of $\cH_2$, namely $\{x\}$. This is a contradiction as $H$ and $[n]\setminus \{x\}$ are incomparable.  We obtained that

\vspace{2mm}

$\bullet$ $\cH_2\subseteq \cD_{S_1\setminus R}\cup \{[n]\}\cup \{\{x\}:x\in R\}$ holds.

\vspace{2mm}

Also, $\cH_3\setminus \cH_3'=\{[n]\setminus \{x\}:x\in R\}$. By definition of a core chain, all sets in $\cH'_3$ are comparable to $S_1$, and since $\cH_2$ contains sets from $B_{\emptyset,S_1}$, we must have $\cH_3\cap B_{\emptyset,S_1}=\emptyset$, thus we obtain

\vspace{2mm}

$\bullet$ $\cH_3\subseteq \cU_{S_1}\cup \{\emptyset\}\cup \{[n]\setminus \{x\}:x\in R\}$.

\bigskip

\textsc{Case I} $\alpha \le 1/2$.

\medskip 

Using Lemma \ref{cubes2} with $a=0$, $b=(1-\alpha)n$, we obtain $\lambda_n(\cD_{S_1}\setminus \{\emptyset\})<1$. Therefore if $\emptyset$ or $[n]$ does not belong to $\cH_2$, then $\lambda_n(\cH_2)<2<1+\sqrt{2}+o(1)$ and we are done. Otherwise 

\ 

$\bullet$ $\lambda_n(\cH_2)\le \lambda_n([n])+\lambda_n(\{\{x\}:x\in R\})+\lambda_n(B_{\emptyset,S_1\setminus R})\le 1+\beta+\frac{1}{1-(\alpha-\beta)}+o(1)$ using Lemma \ref{cubes2} with $a=0$, $b=(1-(\alpha-\beta))n$.

\medskip 

$\bullet$ $\cH_3\subseteq (\cU_{S_1}\setminus \{[n]\}) \cup \{[n]\setminus \{x\}:x\in R\}$ and thus $\lambda_n(\cH_3)\le \frac{1}{\alpha}-1+\beta +o(1)$ using again Lemma \ref{cubes2} this time with $a=\alpha n$, $b=0$.

So $\min \{\lambda_n(\cH_2),\lambda_n(\cH_3)\}\le \min\{1+\beta+\frac{1}{1-(\alpha-\beta)},\frac{1}{\alpha}-1+\beta \}+o(1)$.

\medskip 

\textsc{Case II} $\alpha \ge 1/2$.

\medskip 

It cannot happen that both $\emptyset$ and $[n]$ belong to $\cH_2$ as then $\lambda_n(\cH_3)<\frac{1}{\alpha}-1+\beta+o(1)\le 2+o(1)$. If both $\emptyset$ and $[n]$  belong to $\cH_3$, then by Lemma \ref{cubes2}, we have

\ 

$\bullet$ $\lambda_n(\cH_2)\le \beta+\frac{1}{1-(\alpha-\beta)}-1+o(1)$.

\medskip 

$\bullet$ $\lambda_n(\cH_3)\le \frac{1}{\alpha}+1+\beta +o(1)$.

\ 

In this case, $\min \{\lambda_n(\cH_2),\lambda_n(\cH_3)\}\le \min\{\beta+\frac{1}{1-(\alpha-\beta)}-1,\frac{1}{\alpha}+1+\beta \}+o(1)$.

\medskip

If exactly one of $\emptyset$ and $[n]$ belong to $\cH_2$ and the other to $\cH_3$, then Lemma \ref{cubes2} yields $\min \{\lambda_n(\cH_2),\lambda_n(\cH_3)\}\le \min \{\beta+\frac{1}{1-(\alpha-\beta)},\frac{1}{\alpha}+\beta\}+o(1)$.

\vspace{2mm}

Therefore, the following claim finishes the proof of Theorem \ref{23} \textbf{(ii)}.

\begin{claim}\label{tech} Suppose we have two real numbers $0 \le \beta \le \alpha \le 1$.

\vspace{3mm} 

\textbf{(a)} If $\alpha \le \frac{1}{2},$ then $\min \{1+\beta+\frac{1}{1-(\alpha-\beta)}, \frac{1}{\alpha}-1+\beta\} \le 1 + \sqrt{2}$.
\vspace{2mm}

\textbf{(b)} If $\alpha \ge \frac{1}{2},$ then $\min \{\beta+\frac{1}{1-(\alpha-\beta)}-1,\frac{1}{\alpha}+1+\beta\} \le 1 + \sqrt{2}$.

\textbf{(c)} If $\alpha\ge \frac{1}{2}$, then $\min \{\beta+\frac{1}{1-(\alpha-\beta)},\frac{1}{\alpha}+\beta\}\le 1+\sqrt{2}$.
\end{claim}

\begin{proof} We start with a proposition.

\begin{proposition}\label{ineq1} For $0 \le \beta \le 1/2$ we have 

$$\beta (-\beta^2 + (1 + 2\sqrt{2})\beta -2) \le 0$$

and 

$$
\beta + \frac{2 + \sqrt{2} - \beta}{ - \beta^2 + (1 + \sqrt{2})(\beta + 1)} \le \sqrt{2}.
$$
\end{proposition}

\begin{proof}

Observe that the two inequalities are equivalent. Indeed, as $- \beta^2 + (1 + \sqrt{2})(\beta + 1) > 0$ for $\beta \in [0,1]$, multiplying both sides of the second inequality with $- \beta^2 + (1 + \sqrt{2})(\beta + 1)$ and reorganize it we obtain the first one. 
To see the first inequality, observe that $\beta$ is non-negative while $-\beta^2 + (1 + 2\sqrt{2})\beta -2$ is always negative in case $0\le\beta\le 1/2$, and this proves both statements of the claim.
\end{proof}

Now we continue with the proof of \textbf{(a)}. We prove by contradiction. Let us suppose that we have $\frac{1}{\alpha}-1+\beta > \sqrt{2} +1$
or equivalently (as $\alpha \ge 0$ and $\beta \le 1$)  
\begin{equation}\label{alfa}\alpha < \frac{1}{2 + \sqrt{2}- \beta}.\end{equation}
Our goal is to prove $1+\beta+\frac{1}{1-(\alpha-\beta)} \le 1 + \sqrt{2}$. Observe that if we fix $\beta$, then increasing $\alpha$ would increase $1+\beta+\frac{1}{1-(\alpha-\beta)}$. Thus we have $1+\beta+\frac{1}{1-(\alpha-\beta)}\le 1+\beta+\frac{1}{1-(\frac{1}{2 + \sqrt{2}- \beta}-\beta)}\le 1 +\sqrt{2}$, where the first inequality follows from \eqref{alfa} and the second inequality follows from rearranging $\frac{1}{1-(\frac{1}{2 + \sqrt{2}- \beta}-\beta)}$ to $\frac{2+\sqrt{2}-\beta}{-\beta^2+(1+\sqrt{2})(\beta+1)}$ and applying Proposition \ref{ineq1}.

\medskip 

To prove \textbf{(b)}, first we show that if $\beta>1/2$, then $$\beta+\frac{1}{1-(\alpha-\beta)}-1\le 1 + \sqrt{2}.$$ Indeed, as $\alpha\le 1$, we have $\frac{1}{1-(\alpha-\beta)}<2$, thus $\beta+\frac{1}{1-(\alpha-\beta)}-1<2<1+\sqrt{2}$.

If $\beta\le 1/2$, we proceed similarly to the proof of \textbf{(a)}. Let us suppose indirectly that we have $\frac{1}{\alpha}+1+\beta > 1 + \sqrt{2}$, i.e. $\alpha\le 1/(\sqrt{2}-\beta)$. Our goal is to prove $\beta+\frac{1}{1-(\alpha-\beta)}\le 2+\sqrt{2}$. Observe that if we fix $\beta$, then increasing $\alpha$ would increase $\beta+\frac{1}{1-(\alpha-\beta)}$. Thus we have $$\beta+\frac{1}{1-(\alpha-\beta)}\le \beta+\frac{1}{1-(\frac{1}{\sqrt{2}-\beta}-\beta)}=\beta+\frac{\sqrt{2}-\beta}{-\beta^2+(\sqrt{2}-1)(\beta+1)}.$$ To show $\beta+\frac{\sqrt{2}-\beta}{-\beta^2+(\sqrt{2}-1)(\beta+1)}\le \sqrt{2}+2$, we want to multiply both sides with $-\beta^2+(\sqrt{2}-1)(\beta+1)$. We can do that as it is positive if $0\le\beta\le 1/2$. The resulting inequality $\beta (-\beta^2 + (1 + 2\sqrt{2})\beta -2) \le 0$ holds by Proposition \ref{ineq1}. 

To see \textbf{(c)}, observe that for fixed $\beta$, the first expression is increasing, the second one is decreasing in $\alpha$. Therefore the minimum is maximized when the two expressions are equal, i.e. $\alpha=1-(\alpha-\beta)$ and thus $\alpha=\frac{1+\beta}{2}$. The function $\beta+\frac{2}{1+\beta}$ has maximum $2\le 1+\sqrt{2}$ in the interval $[0,1]$.

\end{proof}
\end{proof}

\medskip

Our final proposition is a construction that shows that $\lim_{n,k}\frac{F'(n,k)}{2^{\lfloor n/2\rfloor}}=\infty$, if both $n$ and $k$ tend to infinity.

\begin{proposition}
For any  integer $k\ge 2$ let $l_k=\lfloor \frac{\log (k-1)}{2}\rfloor$. Then we have $$F'(n,k) \ge (2^{l_k}-o(1))2^{\lfloor n/2\rfloor}.$$
\end{proposition}

\begin{proof}
Fix an integer $k$. If $n$ is large enough, then we can pick $k-1$ sets $F_1,F_2,\dots, F_{k-1}$ of size $\lfloor n/2\rfloor +l_k$ such that $|F_i \cap F_j|\le 0.26n$ for every $i<j\le k-1$ (take sets uniformly at random from the $(\lfloor n/2\rfloor+l_k)$th level of $B_n$). Let us define a coloring $\phi$ by $\phi(A)=k$ if $A\in \cup_{i=1}^{k-1}(\cU_{F_i}\setminus \{F_i\})$ and $\phi(A)=i$ if $A\in \cD_i\setminus \cup_{j=1}^{i-1}\cD_j$ for all $1\le i \le k-1$.

Observe that $\phi$ does not admit a strong rainbow copy of $A_k$ as if $\phi(A)=k$ with $A\in \cU_i$ belongs to a rainbow antichain, then color $i$ is missing. Also, $$|\phi^{-1}(i)|\ge 2^{l_k}2^{\lfloor n/2\rfloor}-(k-2)2^{0.26n}=(2^{l_k}-o(1))2^{\lfloor n/2\rfloor}$$ for all $i=1,2,\dots, k-1$ and $$|\phi^{-1}(k)|\ge (k-1)2^{-l_k}2^{\lfloor n/2\rfloor}-k-\binom{k-1}{2}2^{0.26n}\ge (2^{l_k}-o(1))2^{\lfloor n/2\rfloor}. $$
\end{proof}

\textbf{Funding}: Research supported by the National Research, Development and Innovation Office - NKFIH under the grants K 116769, K 132696, KH 130371, SNN 129364, FK 132060, and KKP-133819, by the J\'anos Bolyai Research Fellowship of the Hungarian Academy of Sciences and the Taiwanese-Hungarian Mobility Program of the Hungarian Academy of Sciences, by Ministry of Science and Technology Project-based Personnel Exchange Program 107 -2911-I-005 -505.

\end{document}